\documentclass[twoside,a4paper,reqno,11pt]{amsart}
\usepackage{amsfonts, amsbsy, amsmath, amssymb, latexsym}
\usepackage{mathrsfs,array}
\usepackage[top=30mm,right=30mm,bottom=30mm,left=30mm]{geometry}
\usepackage{stmaryrd}
\usepackage{bm}

\usepackage{hyperref}

\renewcommand{\a}{\alpha}
\renewcommand{\b}{\beta}

\newcommand{\C}{\mathbb{C}}

\renewcommand{\l}{\lambda} 

 \renewcommand{\to}{\rightarrow}

\newcommand{\leqs}{\leqslant}
\newcommand{\geqs}{\geqslant}

 \newcommand{\vs}{\vspace{3mm}}

\newcommand{\cd}{{\rm cd}}

\makeatletter
\newcommand{\imod}[1]{\allowbreak\mkern4mu({\operator@font mod}\,\,#1)}
\makeatother

\newtheorem{theorem}{Theorem}
\newtheorem*{conj*}{Conjecture}

\newtheorem{corol}[theorem]{Corollary}

\newtheorem{thm}{Theorem}[section]
\newtheorem{prop}[thm]{Proposition}
\newtheorem{lem}[thm]{Lemma}
\newtheorem{cor}[thm]{Corollary}

\theoremstyle{definition}

\begin{document}

\author{Timothy C. Burness}
 \address{T.C. Burness, School of Mathematics, University of Bristol, Bristol BS8 1TW, UK}
 \email{t.burness@bristol.ac.uk}

 \author{Martin W. Liebeck}
\address{M.W. Liebeck, Department of Mathematics,
    Imperial College, London SW7 2BZ, UK}
\email{m.liebeck@imperial.ac.uk}

\author{Aner Shalev}
\address{A. Shalev, Institute of Mathematics, Hebrew University, Jerusalem 91904, Israel}
\email{shalev@math.huji.ac.il}

\title{The length and depth of compact Lie groups}

\begin{abstract}
Let $G$ be a connected Lie group. An unrefinable chain of $G$ is a chain of subgroups $G = G_0 > G_1 > \cdots > G_t = 1$, where each $G_i$ is a maximal connected subgroup of $G_{i-1}$. In this paper, we introduce the notion of the length (respectively, depth) of $G$, defined as the maximal (respectively, minimal) length of such a chain, and we establish several new results for compact groups. In particular, we compute the exact length and depth of every compact simple Lie group, and draw conclusions for arbitrary connected compact Lie groups $G$.
We obtain best possible bounds on the length of $G$ in terms of its dimension, and characterize the connected compact Lie groups that have equal length and depth. The latter result generalizes a well known theorem of Iwasawa for finite groups. More generally, we establish a best possible upper bound on $\dim G'$ in terms of the chain difference of $G$, which is its length minus its depth.  
\end{abstract}

\footnotetext{The second and third authors acknowledge the support of the National Science Foundation under Grant No. DMS-1440140 while they were in residence at the Mathematical Sciences Research Institute in Berkeley, California, during the Spring 2018 semester. The third author acknowledges the support of ISF grant 686/17 and the Vinik chair of mathematics which he holds.}

\subjclass[2010]{}
\date{\today}
\maketitle

\section{Introduction}\label{s:intro}

The maximum length of a chain of subgroups of a finite group $G$ is called the \emph{length} of $G$. This invariant arises naturally in several different contexts and it has been the subject of numerous papers since the 1960s (see \cite{CST,Harada, Jan, Jan2, SST, ST, ST2}, for example). The dual notion of \emph{depth} is defined to be the minimal length of a chain of subgroups
\begin{equation}\label{e:chain}
G = G_0 > G_1 > \cdots > G_t = 1,
\end{equation}
where each $G_i$ is a maximal subgroup of $G_{i-1}$. The depth of finite solvable groups was studied by Kohler \cite{Koh}, and more generally by Shareshian and Woodroofe \cite{SW} in relation to lattice theory. We refer the reader to \cite{BLS,BLS2} for recent work on the length and depth of finite groups and finite simple groups.

In \cite{bls-alg}, we extended these notions to algebraic groups. Let $G$ be a connected algebraic group defined over an algebraically closed field of characteristic $p \geqs 0$. The \emph{length} and \emph{depth} of $G$, denoted by $l(G)$ and $\l(G)$, are defined to be the maximal and minimal length of a chain of subgroups as in \eqref{e:chain}, respectively, where each $G_i$ is a maximal connected subgroup of $G_{i-1}$. The length of $G$ can be computed precisely. Indeed, if $B$ is a Borel subgroup of $\bar{G} = G/R_u(G)$ and $r$ is the semisimple rank of $\bar{G}$, then \cite[Theorem 1]{bls-alg} states that
\begin{equation}\label{e:lg}
l(G) = \dim R_u(G) + \dim B + r.
\end{equation}
For simple algebraic groups, this generalizes a theorem of Solomon and Turull \cite[Theorem A*]{ST2} on finite quasisimple groups. Several results on the depth of algebraic groups are also established in \cite{bls-alg}. For example, if $p=0$ we can calculate the exact depth of every simple algebraic group $G$, obtaining $\l(G) \leqs 6$, with equality if and only if $G$ is of type $A_6$ (see  \cite[Theorem 4]{bls-alg}). We also show that the depth of an algebraic group behaves rather differently in positive characteristic. For example, \cite[Theorem 5(iii)]{bls-alg} states that  the depth of a simple classical type algebraic group tends to infinity with the rank of the group.

In this paper we initiate the study of length and depth for connected Lie groups. Let $\mathfrak{g}$ be a finite dimensional semisimple Lie algebra over $\C$ with compact real form $\mathfrak{g}_0$ (so $\mathfrak{g}_0$ is a compact Lie algebra over $\mathbb{R}$ with $\mathfrak{g}_0 \otimes_{\mathbb{R}} \C \cong \mathfrak{g}$). Under the Lie correspondence, $\mathfrak{g}_0$ is the Lie algebra of a compact semisimple (real) Lie group $G$. We write $G(\C)$ for the corresponding complex Lie group, which can be viewed as a semisimple algebraic group over $\C$. For example, if $\mathfrak{g} = \mathfrak{sl}_n(\mathbb{C})$, then $G = {\rm SU}_{n}$ and $G(\C) = {\rm SL}_{n}(\C)$. Up to isomorphism and isogenies, the compact simple Lie groups are as follows:
\begin{equation}\label{e:simples}
{\rm SU}_{n} \, (n \geqs 2), \; {\rm Sp}_{n} \, (\mbox{$n \geqs 4$ even}), \; {\rm SO}_{n} \, (n \geqs 7),  \; G_2, \; F_4,\; E_6, \; E_7, \; E_8.
\end{equation}
We will write $Cl_n$ to denote any one of the classical type groups ${\rm SU}_{n}$, ${\rm Sp}_{n}$ and ${\rm SO}_{n}$. More generally, it is well known that any compact connected Lie group $G$ is of the form $G = G'Z(G)^0$, where $G'$ is a commuting product of compact simple Lie groups and $Z(G)^0$ is a torus (a direct product of $k$ copies of $T \cong {\rm SO}_2$, the circle group). We will write $T_k$ for a $k$-dimensional torus.

It is natural to transfer the definition of length and depth from algebraic groups to Lie groups. So we define the length $l(G)$ of a connected Lie group $G$ to be the maximal length of a chain of (closed) subgroups as in \eqref{e:chain}, where each $G_i$ is a maximal connected subgroup of $G_{i-1}$ (a sequence of subgroups with this property is called an \emph{unrefinable chain}). Similarly, the depth $\l(G)$ is the minimal length of such a chain. Note that these parameters are independent of any choice of isogeny type. In particular, we may (and will) always assume that if $G$ is a nonabelian compact connected Lie group, then $G' = \prod_{i}S_i$ is a commuting product of simple groups given in \eqref{e:simples}.

We are now ready to state our main results. For a classical type group $G = Cl_n$, define
\begin{equation}\label{e:fgn}
f_G(n) = \left\{\begin{array}{ll}
2n-2 & \mbox{ if $G={\rm SU}_n$} \\
\frac{3}{2}n-1 & \mbox{ if $G = {\rm Sp}_n$} \\
n+\lfloor \frac{n}{4} \rfloor -1 & \mbox{ if $G={\rm SO}_n$.}
\end{array}\right.
\end{equation}
Our first result determines the exact length of every compact simple Lie group.

\begin{theorem}\label{complen}
Let $G$ be a compact simple Lie group.
\begin{itemize}\addtolength{\itemsep}{0.2\baselineskip}
\item[{\rm (i)}] If $G = Cl_n$ is of classical type, then $l(G) = f_G(n)$.
\item[{\rm (ii)}] If $G$ is of exceptional type, then $l(G)$ is as follows:
\[
\begin{array}{cccccc} \hline
G & G_2 & F_4 & E_6 & E_7 & E_8 \\
l(G) & 5&11&13&17&20 \\ \hline
\end{array}
\]
\end{itemize}
\end{theorem}

As an immediate corollary, we deduce that
\[
2r\leqs l(G) < 3r
\]
for all compact simple Lie groups $G$ of rank $r$. In view of \eqref{e:lg}, we also note that
\[
l(G) = O(r^{-1}l(G(\C))).
\]

In fact, it is easy to see that $l(G) \geqs 2r$, independently of Theorem \ref{complen}. Indeed, let $T$ be a maximal torus of $G$, let $\{\a_1, \ldots, \a_r\}$ be a corresponding set of simple roots and consider a chain of Levi subgroups 
\[
G > L_{r-1} > \cdots > L_{1} > T,
\]
where $\{\a_1, \ldots, \a_k\}$ is the set of simple roots in the root system of $L_k$. Then we get    
\[
l(G) \geqs l(T)+r = \dim T + r =2r.
\] 

We can use Theorem \ref{complen} to calculate the length of an arbitrary compact connected Lie group $G$. Recall that $G = G'Z(G)^0$, where the commutator subgroup $G'$ is semisimple (or trivial) and $Z(G)^0$ is a torus (see \cite[Theorem 4.29]{K}, for example).

\begin{theorem}\label{general}
Let $G$ be a compact connected Lie group and write $z = \dim Z(G)^0$, $r = {\rm rank}(G')$ and $G' = \prod_{i=1}^t S_i$, where each $S_i$ is simple. Then
\[
l(G) = z+\sum_{i=1}^{t} l(S_i).
\]
In particular, $z+2r \leqs l(G) \leqs z+3r-t$.
\end{theorem}

We now discuss the relationship between the length of a compact Lie group $G$ and its dimension.
We trivially have $l(G) \leqs \dim G$. The next result shows that $l(G)$ is close to $\dim G$
if and only if $G$ is almost abelian.

\begin{theorem}\label{dimlen}
Let $G$ be a compact connected Lie group.
\begin{itemize}\addtolength{\itemsep}{0.2\baselineskip}
\item[{\rm (i)}] $l(G) = \dim G$ if and only if $G$ is a torus.
\item[{\rm (ii)}] $\dim G - l(G) \leqs \dim G' \leqs 3(\dim G - l(G))$.
\item[{\rm (iii)}] For a collection $\mathcal{C}$ of compact connected Lie groups, the set 
\[
\{ \dim G - l(G) \, : \, G \in \mathcal{C} \}
\]
is bounded if and only if the set 
\[
\{ \dim G - \dim Z(G)^0 \, : \, G \in \mathcal{C} \}
\]
is bounded.
\end{itemize}
\end{theorem}

Note that part (ii) above implies parts (i) and (iii).

The method of proof of Theorem \ref{dimlen} also enables us to characterize algebraic groups
whose length is close to their dimension; see Proposition \ref{alg} below.

As for lower bounds on $l(G)$ in terms of $\dim G$, we show the following.

\begin{theorem}\label{dimlength}
Let $G$ be a compact connected Lie group and let $\a = \sqrt{248}-\sqrt{128}$. Then
\[
l(G) \geqs 5 \cdot 2^{-3/2}(\sqrt{\dim G} - \a).
\]
In particular, $l(G) \geqs (5 \cdot 2^{-3/2}-o(1))\sqrt{\dim G}$, where $o(1) = o_{\dim G}(1)$.
\end{theorem}

This may be compared with Theorem 3 of \cite{bls-alg}, showing that $l(G) > \frac{1}{2} \dim G$
for algebraic groups $G$ over algebraically closed fields.

The proofs of Theorem \ref{dimlen} and \ref{dimlength} rely on the classification of compact simple Lie groups and on Theorems \ref{complen} and \ref{general} above. The first lower bound in Theorem \ref{dimlength} is best possible; indeed it is attained for $G = E_8$.
Similarly, the second lower bound is asymptotically best possible, since
\[
\lim_{n \to \infty} \frac{l({\rm SO}_n)}{\sqrt{\dim {\rm SO}_n}} = 5 \cdot 2^{-3/2}.
\]
In addition, we characterize compact connected Lie groups $G$ of small length as follows.

\begin{theorem}\label{smalll}
Let $c$ be a positive integer and let $G$ be a compact connected Lie group satisfying
\[
l(G) \leqs 5 \cdot 2^{-3/2}(\sqrt{\dim G}+c).
\]
Then either $\dim G$ is $c$-bounded, or $G$ has a normal simple subgroup of type ${\rm SO}_{n}$ of $c$-bounded codimension.
\end{theorem}

Here $c$-bounded means bounded above by some function of $c$ only.

Next we turn to the depth of compact Lie groups. For simple groups, we have the following result (see \cite[Theorem 4]{bls-alg} for the analogous statement for simple algebraic groups over $\C$).

\begin{theorem}\label{liedep}
If $G$ is a compact simple Lie group, then
\[
\l(G) = \left\{\begin{array}{ll}
2 & \mbox{ if $G={\rm SU}_{2}$}  \\
4 & \mbox{ if $G={\rm SU}_{n}\,(n\geqs 4,\,n\ne 7),\,{\rm SO}_{7},\,{\rm SO}_{2r} \, (r \geqs 4),\, E_6$} \\
5 & \mbox{ if $G={\rm SU}_{7}$} \\
3 & \mbox{ in all other cases.}
\end{array}
\right.
\]
In particular, $\l(G) = \l(G(\C))-1$.
\end{theorem}

We do not obtain a precise formula for the depth of an arbitrary compact Lie group, but we give bounds, as follows.

\begin{theorem}\label{depbds}
\mbox{ }
\begin{itemize}\addtolength{\itemsep}{0.2\baselineskip}
\item[{\rm (i)}] Let $S$ be a compact simple Lie group and let $k$ be a positive integer. Then
\[
\l(S^k) = \l(S)+k-1.
\]
\item[{\rm (ii)}]  Let $G$ be a compact connected Lie group and set $z = \dim Z(G)^0$ and $G' = \prod_{i=1}^m S_i^{k_i}$, where the $S_i$ are pairwise non-isomorphic simple groups. Then
\[
z+\sum_{i=1}^m (k_i+1) \leqs \l(G) \leqs z+\sum_{i=1}^m (k_i+\l(S_i)-1).
\]
\end{itemize}
\end{theorem}

By Theorem \ref{liedep}, the upper bound on $\l(G)$ in (ii) is at most $z+1+\sum_{i}(k_i+3)$.

Our next result concerns compact Lie groups $G$ satisfying the condition $l(G) = \l(G)$. Finite groups with this property were characterized by Iwasawa \cite{iw} -- they are precisely the supersolvable groups. For algebraic groups over algebraically closed fields, a partial result was proved in \cite[Theorem 6]{bls-alg}, but not a full characterization. For compact Lie groups, we prove the following result.

\begin{theorem}\label{ld}
Let $G$ be a compact connected Lie group. Then $l(G) = \l(G)$ if and only if $G$ is a torus or $G' = {\rm SU}_2$.
\end{theorem}

More generally, the \emph{chain difference} of $G$ is defined by ${\rm cd}(G) = l(G)-\l(G)$. This invariant has been studied for finite groups and finite simple groups (see \cite{BWZ, BLS2}, for example), with particular interest in groups with chain difference one. Here we determine the compact Lie groups with this property.

\begin{theorem}\label{cd}
Let $G$ be a compact connected Lie group. Then ${\rm cd}(G) = 1$ if and only if
$G' = {\rm SU}_3$, $({\rm SU}_2)^2$ or ${\rm SU}_3{\rm SU}_2$.
\end{theorem}

Our next result bounds the length of $G'$ in terms of the chain difference of $G$.

\begin{theorem}\label{lcd}
Let $G$ be a compact connected Lie group. Then $l(G') \leqs 2 \cd(G) + 2$.
Consequently, $\dim G/Z(G)$ is bounded above by a fixed quadratic function of $\cd(G)$.
\end{theorem}

The above bound on $l(G')$ is best possible. Indeed, if $G = G' = ({\rm SU}_2)^k$ where $k$ is
any positive integer, then $l(G) = 2 \cd(G) +2$. The quadratic function mentioned in the second assertion of Theorem \ref{lcd} is given explicitly
at the end of Section \ref{ss:cd}.

Combining Theorem \ref{lcd} with part (iii) of Theorem \ref{dimlen}, we immediately obtain the following somewhat surprising consequence.

\begin{corol}\label{dimcd}
The following are equivalent for a collection $\mathcal{C}$ of compact connected Lie groups.
\begin{itemize}\addtolength{\itemsep}{0.2\baselineskip}
\item[{\rm (i)}] The set $\{ l(G) - \l(G) \, : \, G \in \mathcal{C} \}$ is bounded.
\item[{\rm (ii)}] The set $\{ \dim G - l(G) \, : \, G \in \mathcal{C} \}$ is bounded.
\item[{\rm (iii)}] The set $\{ \dim G - \l(G) \, : \, G \in \mathcal{C} \}$ is bounded.
\end{itemize}
\end{corol}

Indeed, conditions (i)-(iii) are all equivalent to the set $\{ \dim G - \dim Z(G)^0 \, : \, G \in \mathcal{C} \}$
being bounded.

\vs

We refer the reader to \cite{Sercombe} for results on the length and depth of non-compact Lie groups.

The layout of the paper is as follows. In Section \ref{s:sub} we prove some results on the subgroup structure of compact Lie groups, in particular determining their maximal connected subgroups. Section \ref{s:prel} contains some further preliminary results on lengths and depths of compact groups which are needed for the proofs of the main results. These proofs are given in Section \ref{s:proofs}.

\section{Subgroups of compact Lie groups}\label{s:sub}

The following result provides a close link between the connected subgroups of a compact Lie group $G$ and the connected reductive subgroups of the corresponding complex Lie group. This result is surely well-known, but we have been unable to find a proof in the literature.

\begin{lem}\label{corr}
Let $G(\C)$ be a complex semisimple Lie group, with compact form $G$.
There is a bijective correspondence between $G$-conjugacy classes of connected subgroups of $G$,  and $G(\C)$-conjugacy classes of connected reductive subgroups of $G(\C)$.
\end{lem}

\begin{proof}
Let ${\mathcal X}$ denote the set of connected subgroups of $G$, and ${\mathcal Y}$ the set of connected reductive subgroups in $G(\C)$. For $X \in {\mathcal X}$, let $X^\C \in {\mathcal Y}$ be the complexification of $X$. Define $\phi : {\mathcal X}/G \to {\mathcal Y}/G(\C)$ to be the map that sends the class $X^G$ to the class $(X^\C)^{G(\C)}$. We shall show that $\phi$ is a bijection.

To see that $\phi$ is surjective, let $Y \in {\mathcal Y}$, and let $Y_0$ be a maximal compact subgroup of $Y$. Then $Y_0$ is contained in a maximal compact subgroup $G_0$ of $G(\C)$, which is conjugate to $G$ -- say $G=G_0^g$ with $g \in G(\C)$. Hence $Y_0^g \leqs G$, and $Y^g$ is the complexification of $Y_0^g$. It follows that the class $Y^{G(\C)}$ is in the image of $\phi$.

It remains to show that $\phi$ is injective. Let $X_1,X_2 \in {\mathcal X}$, and suppose $X_1^\C$ and $X_2^\C$ are conjugate in $G(\C)$, say $X_2^\C = (X_1^\C)^g$. Then $X_2$ and $X_1^g$ are maximal compact subgroups of $X_2^\C$, hence are conjugate in $X_2^\C$, so $X_2=X_1^h$ with $h \in  X_2^\C \leqs G(\C)$. Now consider the Cartan decomposition of $G(\C)$:
\[
G(\C) = GP
\]
(see \cite[p.446]{K}). Write $h = kp$ with $k \in G, p \in P$. Then for $x_1 \in X_1$, letting $x_2 = x_1^h$ we have
\[
x_1kp = kx_2p^{x_2}.
\]
Now $x_1k, kx_2 \in G$ and $p,p^{x_2} \in P$, so by uniqueness in the Cartan decomposition, $x_1k = kx_2$. It follows that $X_1^k=X_2$, so $X_1$ and $X_2$ are conjugate in $G$. This proves the injectivity of $\phi$, as required.
\end{proof}

Since the complexification map $X \to X^\C$ is inclusion-preserving, the following is an immediate consequence of Lemma \ref{corr}.

\begin{cor}\label{maxs}
There is a bijective correspondence between conjugacy classes of maximal connected subgroups in $G$  and conjugacy classes of maximal connected reductive subgroups in $G(\C)$.
\end{cor}

Here, by  maximal connected reductive subgroups in $G(\C)$, we mean subgroups that are maximal among connected reductive subgroups of $G(\C)$.

In the next two results we use Corollary \ref{maxs}, together with known results on the subgroup structure of $G(\C)$, to describe the maximal connected subgroups of compact simple Lie groups. For subgroups of maximal rank, this was done by Borel and de Siebenthal \cite{BS}.

\begin{prop}\label{class}
Let $G = Cl_n$ be a compact simple Lie group of classical type with natural module $V$ of dimension $n$ over $\C$, and let $M$ be a maximal connected subgroup of $G$. Then one of the following holds:
\begin{itemize}\addtolength{\itemsep}{0.2\baselineskip}
\item[{\rm (i)}] $M$ is a reducible or tensor product subgroup of $G$, as listed in Table $\ref{redtens}$;
\item[{\rm (ii)}] $G={\rm SU}_n$ and $M = {\rm Sp}_n$ (with $n \geqs 4$ even) or ${\rm SO}_{n}$;
\item[{\rm (iii)}] $M$ is a  compact simple Lie group acting irreducibly on $V$, and $M$ is not isomorphic to a classical group on $V$.
\end{itemize}
\end{prop}

\begin{proof}
Let $G(\C) = Cl(V)$ be the corresponding complex simple classical group. Elementary considerations (see the first two paragraphs of the proof in \cite[p.279]{Sei}) show that the maximal connected subgroups of $G(\C)$ are among the following:
\begin{itemize}\addtolength{\itemsep}{0.2\baselineskip}
\item[(a)] parabolic subgroups,
\item[(b)] stabilizers of non-degenerate subspaces of $V$,
\item[(c)] tensor product subgroups of the form $Cl_a(\C) \otimes Cl_b(\C)$, where $n=ab$,
\item[(d)] classical subgroups ${\rm Sp}_n(\C),{\rm SO}_n(\C)$ of  $G(\C) = {\rm SL}_n(\C)$,
\item[(e)] simple subgroups acting irreducibly on $V$.
\end{itemize}
Hence the maximal connected reductive subgroups of $G(\C)$ are those in (b)--(e), together with those Levi subgroups that are maximal among connected reductive subgroups. By \cite{BS} (see the table on p.219), the only such Levi subgroups are of the form ${\rm SL}_{n/2}(\C)T_1$ in $G(\C) = {\rm Sp}_n(\C)$ or ${\rm SO}_n(\C)$. Now the conclusion follows from Corollary \ref{maxs}.
\end{proof}

\begin{table}[h]
\[
\begin{array}{lll} \hline
G & \mbox{Reducible subgroups} & \mbox{Tensor product subgroups ($n=ab$)} \\ \hline
{\rm SU}_n & ({\rm SU}_k \times {\rm SU}_{n-k})T_1 & {\rm SU}_a \otimes {\rm SU}_b \\
{\rm Sp}_n & {\rm Sp}_k\times {\rm Sp}_{n-k},\;  {\rm SU}_{n/2}T_1 & {\rm Sp}_a\otimes {\rm SO}_b \\
{\rm SO}_n & {\rm SO}_k\times {\rm SO}_{n-k},\; {\rm SU}_{n/2}T_1\; (n \mbox{ even}) & {\rm SO}_a\otimes {\rm SO}_b,\; {\rm Sp}_a\otimes {\rm Sp}_b \\ \hline
\end{array}
\]
\caption{The maximal connected reducible and tensor product subgroups of compact simple Lie groups of classical type}
\label{redtens}
\end{table}

\begin{prop}\label{ex}
Let $G$ be a compact simple Lie group of exceptional type and let $M$ be a maximal connected subgroup of $G$. Then the possibilities for $M$ are listed in Table \ref{tab:exmax}.
\end{prop}

\begin{proof}
The results of Dynkin \cite{Dynkin2} show that the maximal connected subgroups of the complex simple group $G(\C)$ are parabolic subgroups, together with subgroups as in Table \ref{tab:exmax}, but excluding the subgroups $D_5T_1<E_6$ and $E_6T_1 < E_7$ (these are of course Levi factors of parabolics). By \cite{BS}, these two Levi subgroups are the only ones that are maximal among connected reductive subgroups of $G(\C)$.
Hence the maximal connected reductive subgroups of $G(\C)$ are precisely those in the table, and the result now follows via Corollary \ref{maxs}.
\end{proof}

\begin{table}[h]
\[
\begin{array}{ll} \hline
G & M \\ \hline
G_2 & A_2,\,A_1^2,\,A_1 \\
F_4 & B_4,\,C_3A_1,\,A_2^2,\,A_1G_2,\,A_1  \\
E_6 & D_5T_1,\,A_5A_1,\,A_2^3,\,F_4,\,C_4,\,A_2G_2, \,G_2,\,A_2 \\
E_7 & D_6A_1,\,A_5A_2,\,A_7,\,E_6T_1,\, G_2C_3, \, F_4A_1,\,A_1^2,\,A_2,\,A_1 \\
E_8 & E_7A_1,\,E_6A_2,\,D_8,\,A_8,\,A_4^2, \, G_2F_4,\,A_2A_1,\,B_2,\,A_1  \\ \hline
\end{array}
\]
\caption{The maximal connected subgroups of compact simple Lie groups of exceptional type}
\label{tab:exmax}
\end{table}

\section{Length and depth: preliminary results}\label{s:prel}

We start by recording some elementary properties of the length and depth of connected Lie groups, which we will use repeatedly throughout the paper. Note that the proof of \cite[Lemma 2.1]{CST} goes through to give part (i).

\begin{lem}\label{l:easy}
Let $G$ be a connected Lie group with a connected normal subgroup $N$.
\begin{itemize}\addtolength{\itemsep}{0.2\baselineskip}
\item[{\rm (i)}] $l(G) = l(N) + l(G/N)$.
\item[{\rm (ii)}] $\l(G/N) \leqs  \l(G) \leqs \l(N) + \l(G/N)$.
\item[{\rm (iii)}] $l(G) = 1+ \max\{l(M) \, : \, \mbox{$M$ maximal connected in $G$} \}$.
\item[{\rm (iv)}] $\l(G) = 1+ \min\{\l(M) \, : \, \mbox{$M$ maximal connected in $G$} \}$.
\end{itemize}
\end{lem}

\begin{lem}\label{triv}
If $G$ is a connected compact Lie group, then
\[
\l(G) = \l(G')+\dim Z(G)^0.
\]
\end{lem}

\begin{proof}
We proceed by induction on the dimension of $G$. The conclusion is clear if $G$ is a torus (every maximal connected subgroup has codimension $1$), so assume $G=G'Z$ with $G'\ne 1$ and $Z = Z(G)^0$. Set $z = \dim Z(G)^0$ and let $M$ be a maximal connected subgroup of $G$ such that $\l(M) = \l(G)-1$.

If $G'\leqs M$, then $M = G'Z_0$ with $Z_0$ of codimension 1 in $Z$, and induction gives
\[
\l(M) = \l(G') + \dim Z_0 = \l(G')+z-1,
\]
hence $\l(G) = \l(G')+z$.

Now assume $G'\not\leqs M$. Then $M = M_1Z$, where $M_1$ is maximal connected in $G'$.
By induction, $\l(M) = \l(M_1')+z+z_1$, where $z_1 = \dim Z(M_1)^0$. Now $\l(M_1) \leqs \l(M_1')+z_1$ by Lemma \ref{l:easy}(ii), so $\l(M) \geqs \l(M_1)+z$. By
the definition of depth we have $\l(G') \leqs \l(M_1)+1$. Hence $\l(G) = \l(M)+1 \geqs \l(G')+z$. Finally, Lemma \ref{l:easy}(ii) implies that
$\l(G) \leqs \l(G')+z$
and the proof is complete.
\end{proof}

\begin{lem}\label{l:depth12}
Let $G$ be a connected compact Lie group.
\begin{itemize}\addtolength{\itemsep}{0.2\baselineskip}
\item[{\rm (i)}] $\l(G)=1$ if and only if $G = T_1$.
\item[{\rm (ii)}] $\l(G)=2$ if and only if $G = T_2$ or ${\rm SU}_{2}$.
\end{itemize}
\end{lem}

\begin{proof}
Part (i) is obvious, so let us consider (ii). If $G=T_2$ then $\l(G)=2$ by Lemma \ref{triv}; and if $G={\rm SU}_2$, then $G$ has a maximal connected subgroup $T_1$ by Corollary \ref{maxs}, so $\l(G)=2$ again. Conversely, suppose
$G$ is a connected compact Lie group with $\l(G)=2$. By Lemma \ref{triv}, either $G=T_2$ or $G=G'$. In the latter case, $G$ has a maximal connected subgroup $T_1$ by part (i), and it follows that $G={\rm SU}_2$. This completes the proof. 
\end{proof}

\begin{lem}\label{corr1}
Let $G(\C)$ be a complex semisimple Lie group, with compact form $G$.
Then  $l(G) < l(G(\C))$.
\end{lem}

\begin{proof}
It follows from Lemma \ref{corr} that $l(G) \leqs l(G(\C))$. To see that the inequality is strict, we apply \cite[Corollary 2]{bls-alg}, which states that every unrefinable chain of $G(\C)$ of maximum length includes a maximal parabolic subgroup.
\end{proof}

In the statement of the next result, we refer to the function defined in \eqref{e:fgn}.

\begin{lem}\label{fgbd}
Let $H = Cl_k$ be a compact simple Lie group of classical type and let $N = N(H,k)$ be the minimal dimension of a nontrivial irreducible complex representation of the simply connected cover of $H$ such that $N>k$ and $H$ is not isomorphic to a classical group $Cl_N$. Then for any simple Lie group $G = Cl_N$, we have $f_G(N) > f_H(k)$.
\end{lem}

\begin{proof} The complex irreducible representations of $H$ are the restrictions of those of the corresponding complex simple group $Cl_k(\C)$, which are parametrized by dominant weights. Hence
the values of $N(H,k)$ can be read off from \cite{lub}, and they are listed in Table \ref{tab:lub}. By definition of the function $f_G$, for $G = Cl_N$ we have $f_G(N) \geqs \frac{5}{4}N-1$ and $f_H(k) \leqs 2k-2$. For $H \ne {\rm SO}_{7}$, it is routine to check that $\frac{5}{4}N-1 > 2k-2$ and thus $f_G(N) > f_H(k)$. Finally, if $H = {\rm SO}_{7}$, then $N = 8$ and $f_G(N) \geqs 9$, which is greater than $f_H(k)=7$. The result follows.
\end{proof}

\begin{table}
\[
\begin{array}{lll} \hline
H & \hspace{5mm} N(H,k) & \\ \hline
{\rm SU}_{k} & \left\{\begin{array}{l}
4, \, 6, \, 10   \\
\frac{1}{2}k(k-1)
\end{array}\right.
& \begin{array}{l}
k=2, \, 3, \, 4 \mbox{ (resp.)} \\
k>4
\end{array} \\
{\rm Sp}_{k} & \left\{\begin{array}{l}
10  \\
\frac{1}{2}k(k-1) -1
\end{array}\right.
&
\begin{array}{l}
k=4  \\
k>4
\end{array} \\
{\rm SO}_{k} & \left\{\begin{array}{l}
2^{\lfloor \frac{k-1}{2}\rfloor}  \\
\frac{1}{2}k(k-1)
\end{array}\right.
&
\begin{array}{l}
7 \leqs k \leqs 14,\, k \ne 8  \\
\mbox{$k = 8$ or $k >14$}
\end{array} \\ \hline
\end{array}
\]
\caption{The values of $N(H,k)$ in Lemma \ref{fgbd}}
\label{tab:lub}
\end{table}

\begin{lem}\label{exbd}
Let $H$ be a compact simple Lie group of exceptional type, and let $N = N(H)$ be the minimal dimension of a nontrivial complex representation of the simply connected cover of $H$. Then
\[
l(H) < \min\{f_G(N) \,:\, G = Cl_N\}.
\]
\end{lem}

\begin{proof}
The values of $N=N(H)$ and $m = \min\{f_G(N) \,:\, G = Cl_N\}$ are as follows:
\[
\begin{array}{lccccc} \hline
H & G_2 & F_4 & E_6 & E_7 & E_8  \\
N(H) & 7 & 26 & 27 & 56 & 248 \\
m & 7 & 31 & 32 & 69 & 309 \\
\hline
\end{array}
\]
The result for $H=E_8$ follows immediately (since $m > \dim E_8$).
In the remaining cases, choose a maximal connected subgroup $M$ in $H$ such that $l(M) = l(H)-1$. We need to show that $l(M) \leqs m-2$. This is obvious if $\dim M \leqs m-2$, so we may assume that $\dim M \geqs m-1$ and we can use Lemma \ref{ex} to read off the possibilities for $M$. We then combine Lemmas \ref{l:easy}(i) and \ref{corr1} with the expression in \eqref{e:lg} to produce an upper bound on $l(M)$. For example, if $H = E_7$ and $M = D_6A_1$ then
\[
l(M) \leqs l(D_6) + l(A_1) \leqs l(D_6(\C)) + l(A_1(\C)) - 2 = 43 \leqs m-2.
\]
In this way, it is easy to check that $l(M) \leqs m-2$ in all cases, unless $(H,M) = (E_6,F_4)$ or $(G_2,A_2)$. In the first case, by considering the maximal connected subgroups of $M = F_4$, we deduce that $l(M) \leqs 24$. Similarly, $l(A_2) \leqs 4$ and the result follows.
\end{proof}

\begin{lem}\label{lowerbd}
If $G = Cl_n$ is a compact simple classical Lie group, then $l(G) \geqs f_G(n)$.
\end{lem}

\begin{proof}
For $G = {\rm SU}_n$, we have the following unrefinable chain of connected subgroups of length $f_G(n)=2n-2$:
\[
{\rm SU}_n > {\rm SU}_{n-1}T_1 > {\rm SU}_{n-1} > \cdots > {\rm SU}_2 > T_1 > 1.
\]
Similarly, if $G = {\rm Sp}_n$ with $n=2k$, there is an unrefinable chain
\[
{\rm Sp}_{2k} > {\rm Sp}_2 \times {\rm Sp}_{2k-2} > ({\rm Sp}_2)^2 \times {\rm Sp}_{2k-4} > \cdots > ({\rm Sp}_2)^k
\]
of length $k-1$. Since ${\rm Sp}_{2} \cong {\rm SU}_{2}$ has length $2$, it follows that $l(({\rm Sp}_2)^k)=2k$ and thus $l(G) \geqs 3k-1 = f_G(n)$.

Finally, consider $G = {\rm SO}_n$. Suppose first that $n$ is not divisible by 4, and write $n=4k+s$ with $1\leqs s\leqs 3$. There is an unrefinable chain
\[
{\rm SO}_n > {\rm SO}_4 \times {\rm SO}_{n-4} > \cdots > ({\rm SO}_4)^k\times {\rm SO}_s
\]
of length $k$. Since ${\rm SO}_{4} \cong ({\rm SU}_{2} \times {\rm SU}_{2})/Z_2$, ${\rm SO}_{3} \cong {\rm SU}_{2}/Z_2$ and ${\rm SO}_{2} \cong T_1$ (where $Z_2$ is a cyclic group of order $2$), we see that $l(({\rm SO}_4)^k \times {\rm SO}_s)= 4k+s-1$ and the result follows.
Similarly, if $n = 4k$, then the above chain ${\rm SO}_n > \cdots >  ({\rm SO}_4)^k$ has length $k-1$ and thus $l(G) \geqs 5k-1 = f_G(n)$.
\end{proof}

\section{Proofs of the main results}\label{s:proofs}

\subsection{Proof of Theorem \ref{complen}}

First assume $G = Cl_n$ is a compact simple Lie group of classical type.
We prove that $l(G) \leqs f_G(n)$ by induction on $n$. In view of Lemma \ref{lowerbd}, this will complete the proof of part (i) of Theorem \ref{complen}. Choose a maximal connected subgroup $M$ of $G$ with $l(M)=l(G)-1$; we need to show that $l(M) < f_G(n)$. The possibilities for $M$ are given by Proposition \ref{class}.

Consider first the reducible subgroups in Proposition \ref{class}(i). If $G={\rm SU}_n$, then $M = ({\rm SU}_k\times {\rm SU}_{n-k})T_1$ and by induction we deduce that
\[
l(M) \leqs (2k-2)+2(n-k)-2+1 = 2n-3 < f_G(n)
\]
as required. A similar argument applies when $G$ is a symplectic or orthogonal group. For example, if $G = {\rm Sp}_{n}$ and $M = {\rm SU}_{n/2}T_1$, then 
\[
l(M) = l({\rm SU}_{n/2})+1 \leqs n-1 < f_G(n).
\]
Similarly, if $G={\rm SO}_n$ and $M = {\rm SO}_k\times {\rm SO}_{n-k}$, then induction gives
\[
l(M) \leqs k+\left\lfloor \frac{k}{4}\right\rfloor -1 +n-k+\left\lfloor \frac{n-k}{4}\right\rfloor-1 < f_G(n).
\]

Next suppose that $M$ is a tensor product subgroup, as in Proposition \ref{class}(i). Here induction clearly gives $l(M) < f_G(n)$, except possibly in the case where $G={\rm SO}_n$ and $M = {\rm Sp}_a \otimes {\rm Sp}_b$, with $n=ab>4$. Here induction yields
\[
l(M) \leqs \frac{3}{2}(a+b)-2 \leqs \frac{5}{4}ab-2 = f_G(n)-1
\]
as required.

Suppose $G = {\rm SU}_n$ and $M={\rm Sp}_n$ or ${\rm SO}_{n}$, as in Proposition \ref{class}(ii). If $n=2$ then $M = {\rm SO}_{2}$ and $l(M) = 1 < f_G(n)=2$. For $n \geqs 3$, induction gives $l(M) \leqs \frac{3}{2}n-1 <2n-2 = f_G(n)$.

Finally, suppose $M$ is as in Proposition \ref{class}(iii). Here $M$ is simple and acts irreducibly on the natural module $V$. Also $M$ is not isomorphic to a classical group on $V$. If $M$ is classical, say $M = Cl_k$, then induction gives $l(M) \leqs f_M(k)$ and Lemma \ref{fgbd} shows that this is less than $f_G(n)$. Similarly, if $M$ is of exceptional type, Lemma \ref{exbd} gives $l(M) < f_G(n)$. This completes the proof of part (i) of Theorem \ref{complen}.

\vs

Now assume $G$ is of exceptional type. As before, we choose a maximal connected subgroup $M$ of $G$ with $l(M) = l(G)-1$. The possibilities for $M$ are recorded in Table \ref{tab:exmax}. By applying part (i), it is an easy exercise, starting with the case $G = G_2$ and working down the rows of the table, to compute $l(M)$ in every case. In this way, we obtain the values for $l(G)$ recorded in part (ii) of the theorem. For example, if $G=F_4$ then by taking $M = B_4$ we can construct a chain
\[
F_4 > B_4 > B_2A_1^2 > A_1^4 > A_1^3T_1 > A_1^3 > A_1^2T_1> A_1^2 > A_1T_1 > A_1 > T_1 > 1
\]
of length $11$. Similarly, in the other cases we take 
\[
(G,M) = (G_2,A_2), (E_6,D_5T_1), (E_7,D_6A_1), (E_8,D_8)
\]
to build chains of length $5,13,17$ and $20$, respectively.

\subsection{Proof of Theorem \ref{general}}

As in the statement of the theorem, let $G$ be a compact connected Lie group and write $G = G'Z(G)^0$, where $G' = \prod_{i=1}^{t}S_i$ is a commuting product of simple groups. Then Lemma \ref{l:easy}(i) implies that $l(G) = z+\sum_i l(S_i)$, where $z = \dim Z(G)^0$. Let $r_i = {\rm rank}(S_i)$ and $r = \sum_{i} r_i = {\rm rank}(G')$. Then $2r_i \leqs l(S_i) \leqs 3r_i-1$ by  Theorem \ref{complen}, hence
\[
z+2r = z+\sum_{i} 2r_i \leqs l(G) \leqs z+\sum_{i}(3r_i-1) = z+3r-t
\]
as required.

\subsection{Proof of Theorem \ref{dimlen}}

Let $G = G'Z(G)^0$ be a compact connected Lie group.

\begin{lem}\label{simple}
If $G$ is simple, then $l(G) \leqs \frac{2}{3} \dim G$, with equality if and only if $G = {\rm SU}_2$.
\end{lem}

\begin{proof}
This follows easily from Theorem \ref{complen}.
\end{proof}

Denote $\Delta(G) = \dim G - l(G)$. The additivity of $\dim$ and $l$ implies the additivity of $\Delta$, namely $\Delta(G) = \Delta(G/N) + \Delta(N)$, where $N$ is a connected normal subgroup of $G$.

\begin{lem}\label{semisimple}
If $G$ is semisimple, then $\Delta(G) \leqs \dim G \leqs 3 \Delta(G)$.
\end{lem}

\begin{proof}
The first inequality is trivial for any compact connected Lie group $G$, and the second reduces to the case where $G$ is simple, by additivity.
For $G$ simple we have $3 l(G) \leqs 2 \dim G$ by Lemma \ref{simple}, which yields
$\dim G \leqs 3 \Delta(G)$ as required.
\end{proof}

\begin{lem}\label{reduce}
We have $\Delta(G) = \Delta(G')$.
\end{lem}

\begin{proof}
This is clear since $\Delta(G) = \Delta(G') + \Delta(Z(G)^0)$ and $\Delta(Z(G)^0) = 0$.
\end{proof}

We can now prove Theorem \ref{dimlen}. It suffices to prove part (ii), namely
\[
\Delta(G) \leqs \dim G' \leqs 3 \Delta(G).
\]
In view of Lemma \ref{reduce}, this reduces to $\Delta(G') \leqs \dim G' \leqs 3 \Delta(G')$, which is Lemma \ref{semisimple} (since $G'$ is semisimple). This completes the proof of the theorem.

\vspace{2mm}

A similar method enables us to characterize connected algebraic groups $G$ over algebraically closed fields, of large length.
We clearly have $l(G) \leqs \dim G$, and by \cite[Theorem 3]{bls-alg} equality holds if and only
if $G = R(G).A_1^t$ for some $t \geqs 0$, where $R(G)$ is the radical of $G$ (and the extension is not necessarily split). Here we extend this
by showing that $\dim G - l(G)$ is bounded if and only if the codimension of $R(G).A_1^t$
is bounded, where $t$ is the multiplicity of $A_1$ in the semisimple group $G/R(G)$.
More precisely, we prove the following.

\begin{prop}\label{alg}
Let $G$ be a connected algebraic group over an algebraically closed field. Set $\Delta(G) = \dim G - l(G)$
and let $t \geqs 0$ be as above. Then 
\[
\Delta(G) \leqs \dim (G/R(G).A_1^t) \leqs 8\Delta(G).
\]
\end{prop}

\begin{proof}
If $G$ is simple and not isomorphic to $A_1$, then by applying \cite[Corollary 2]{bls-alg} we deduce that
\[
l(G) \leqs \frac{7}{8}\dim G,
\]
with equality if and only if $G = A_2$, and this implies
\[
\Delta(G) \leqs \dim G \leqs 8\Delta(G).
\]
For an arbitrary connected algebraic group $G$, write $G/R(G) = A_1^tS_1 \cdots S_k$, where $t,k \geqs 0$ and each $S_i$ is a simple algebraic group that is not isomorphic to $A_1$. Since $\Delta$ is additive and
$\Delta(R(G)) = \Delta(A_1) = 0$, we conclude that $\Delta(G) \leqs \dim (G/R(G).A_1^t) \leqs 8\Delta(G)$, as required.
\end{proof}

\subsection{Proof of Theorem \ref{dimlength}}\label{ss:dl}

We first express the length of a simple classical compact Lie group in terms of its dimension.

\begin{lem}\label{lendim}
Let $S = Cl_n$ be a simple classical compact Lie group and let $d = \dim S$. Then
\[
l(S) = \left\{\begin{array}{ll}
2 \sqrt{d+1} - 2 & \mbox{if $S = {\rm SU}_{n}$} \\
3 \cdot 2^{-1/2} \cdot \sqrt{d+\frac{1}{8}}  - \frac{7}{4} & \mbox{if $S = {\rm Sp}_{n}$} \\
5 \cdot 2^{-3/2} \cdot \sqrt{d+\frac{1}{8}} - \frac{2k+3}{8} & \mbox{if $S = {\rm SO}_n$,}
\end{array}\right.
\]
where $n \equiv k \imod{4}$ and $0 \leqs k \leqs 3$.
\end{lem}

\begin{proof}
This follows from Theorem \ref{complen} and the well known formulae for the dimensions of the relevant groups.
\end{proof}

\begin{cor}\label{limit}
Let $S$ and $d$ be as in Lemma $\ref{lendim}$.
\begin{itemize}\addtolength{\itemsep}{0.2\baselineskip}
\item[{\rm (i)}] $l(S) \geqs 5 \cdot 2^{-3/2} \cdot \sqrt{d} - \frac{9}{8}$ in all cases.
\item[{\rm (ii)}] We have 
\[
\lim_{d \to \infty} \frac{l(S)}{\sqrt{d}} = 
\left\{\begin{array}{ll}
2 & \mbox{ if $S = {\rm SU}_n$} \\
3 \cdot 2^{-1/2} & \mbox{ if $S={\rm Sp}_n$} \\
5 \cdot 2^{-3/2} & \mbox{ if $S = {\rm SO}_n$.}
\end{array}
\right.
\]
\end{itemize}
\end{cor}

The next result also deals with exceptional groups.

\begin{lem}\label{simplel}
Let $S$ be a compact simple Lie group and set
\begin{equation}\label{e:al}
\a = \sqrt{248} - \sqrt{128} = 4.4343...
\end{equation}
Then
\[
l(S) \geqs 5 \cdot 2^{-3/2}(\sqrt{\dim G} - \xi),
\]
where $\xi = \a$ if $S = E_6, E_7, E_8$, otherwise $\xi=1$.
\end{lem}

\begin{proof}
If $S$ is classical, this follows from Corollary \ref{limit} (since $5 \cdot 2^{-3/2} = 1.7677... > 9/8$). For $S$ exceptional, the lower bound follows from Theorem \ref{complen}.
\end{proof}

Note that the inequality in Lemma \ref{simplel} is essentially best possible. For example,
\[
l(E_8) =  20 = 5 \cdot 2^{-3/2}(\sqrt{\dim E_8} - \a)
\]
and
\[
l(F_4) = 11 > 5 \cdot 2^{-3/2}(\sqrt{\dim F_4} - 1) = 10.9797...
\]

Set
\begin{equation}\label{e:be}
\b = 5 \cdot 2^{-3/2} = 1.7677...
\end{equation}
and let $\a$ be as in Lemma \ref{simplel}.

We will need the following elementary inequalities.

\begin{lem}\label{elem}
Let $x, y$ be real numbers.
\begin{itemize}\addtolength{\itemsep}{0.2\baselineskip}
\item[{\rm (i)}] If $x \geqs 1$ then $1+ \b\sqrt{x} \geqs \b \sqrt{x+1}$.
\item[{\rm (ii)}] If $x, y \geqs 3$ then $\sqrt{x} + \sqrt{y} \geqs \sqrt{x+y} + 1$.
\item[{\rm (iii)}] If $x, y \geqs 78$ then $\sqrt{x} + \sqrt{y} \geqs \sqrt{x+y} + \a$.
\end{itemize}
\end{lem}

\begin{proof}
This is easily reduced to quadratic inequalities in one variable.
For example, let us prove part (ii). By squaring both sides of the required inequality
we reduce it to $\sqrt{xy} \geqs \sqrt{x+y} + \frac{1}{2}$, namely to $xy \geqs x+y + \sqrt{x+y} + \frac{1}{4}$.
Since $(x-3)(y-3) \geqs 0$ we obtain $xy \geqs 3(x+y) - 9$ so it suffices to show that
$2(x+y) \geqs \sqrt{x+y} + 9.25$. Let $w = \sqrt{x+y}$. Then we have to
show that $2w^2 - w - 9.25 \geqs 0$, which follows from the fact that
$w \geqs \sqrt{6} > (1+ \sqrt{75})/4$.
\end{proof}

Let $G$ be a compact connected Lie group. We prove Theorem \ref{dimlength} by induction on $\dim G$, the base case $\dim G = 1$
being trivial. Suppose $\dim G > 1$ and let $N \ne 1$ be a connected normal subgroup of $G$ of minimal dimension. If $Z(G)^0 \ne 1$ then $N$ is a $1$-dimensional torus.
Set $x = \dim G/N$. By the induction hypothesis, $l(G/N) \geqs \b(\sqrt{x}-\a)$ and by applying
part (i) of Lemma \ref{elem} we obtain
\[
l(G) = 1 + l(G/N) \geqs 1 + \b(\sqrt{x}-\a) \geqs \b(\sqrt{x+1}-\a) = \b (\sqrt{\dim G} - \a),
\]
as required.

We therefore may assume $Z(G)^0 = 1$, so $G = \prod_{i=1}^t S_i$ is semisimple.
If $t = 1$ then $G$ is simple and the result follows from Lemma \ref{simplel}.
So suppose $t \geqs 2$.
We may assume $N = S_1$, and set $x = \dim G/N $ and $y = \dim N$.
Then $x \geqs y \geqs 3$. Suppose first that $N \ne E_6, E_7, E_8$. Then the
induction hypothesis, Lemma \ref{simplel} and part (ii) of Lemma \ref{elem} yield
\[
l(G) = l(G/N) + l(N) \geqs \b(\sqrt{x}- \a) + \b(\sqrt{y}-1) \geqs \b(\sqrt{x+y} - \a)
= \b(\sqrt{\dim G} - \a).
\]
It remains to deal with the case where $N = S_1$ is $E_6, E_7$ or $E_8$.
Then $x \geqs y \geqs 78$. Combining part (iii) of Lemma \ref{elem} with the induction
hypothesis we obtain
\[
l(G) \geqs  \b(\sqrt{x}- \a) + \b(\sqrt{y}-\a) \geqs \b(\sqrt{x+y} - \a)
= \b(\sqrt{\dim G} - \a).
\]
This completes the proof.

\subsection{Proof of Theorem \ref{smalll}}

Let the constants $\a$ and $\b$ be as defined above in \eqref{e:al} and \eqref{e:be}.
Let $G$ be a compact connected Lie group with
\[
l(G) \leqs \b(\sqrt{\dim G}+c).
\]

Let $N$ be a connected normal subgroup of $G$. If $l(G/N) > \b(\sqrt{\dim G/N}+c+\a)$ then
\[
l(G) > \b(\sqrt{\dim G}+c) + (l(N) - \b(\sqrt{\dim N}-\a)) \geqs \b(\sqrt{\dim G}+c)
\]
since $l(N) \geqs \b(\sqrt{\dim N}-\a)$ by Theorem \ref{dimlength}. This is a contradiction, so
\begin{equation}\label{e:sq}
l(G/N) \leqs \b(\sqrt{\dim G/N}+c+\a).
\end{equation}
In particular, if $z = \dim Z(G)^0$ then $z < \b(\sqrt{z}+c+\a)$ and thus $z$ is $c$-bounded. Therefore, we may  assume that $G$ is semisimple.

Suppose $G$ has $t$ factors of type $E_8$, so there is a connected normal subgroup $N$ with $G/N = (E_8)^t$. By \eqref{e:sq}, we have
\[
20t \leqs \b(\sqrt{248t}+c+\a)
\]
and we deduce that $t$ is $c$-bounded. In the same way, we see that the number of exceptional factors of $G$ is $c$-bounded so we may assume that $G$ is a product of classical groups. If $G/N = ({\rm SU}_{n})^t$ for some $n$ and $t$, then \eqref{e:sq} implies that
\[
2t(n-1) \leqs \b(\sqrt{t(n^2-1)}+c+\a)
\]
and thus $n$ and $t$ are $c$-bounded. The same conclusion holds if $G/N = ({\rm Sp}_{n})^t$.

Therefore, to complete the proof we may assume that $G = \prod_{i=1}^{k}{\rm SO}_{n_i}$, where $k \geqs 2$ and $n_1 \geqs n_{i} \geqs 7$ for all $i$.  By Theorem \ref{complen}, we have
\begin{align*}
l(G) & = \sum_{i=1}^{k}n_i + \sum_{i=1}^{k}\lfloor n_i/4 \rfloor - k \geqs \frac{5}{4}\sum_{i=1}^{k}n_i -\frac{7}{4}k, \hbox{ and} \\
\b\sqrt{\dim G} & = \frac{5}{4}\sqrt{\sum_{i=1}^{k}n_i(n_i-1)}.
\end{align*}

\noindent \emph{Claim.} If $(n_1, \ldots, n_k)$ is a $k$-tuple of integers with $k \geqs 2$ and $n_1 \geqs n_i \geqs 7$ for all $i$, then either $(n_1,k) = (7,2)$, or
\begin{equation}\label{e:sum}
\frac{5}{4}\sum_{i=1}^{k}n_i -\frac{7}{4}k - \frac{5}{4}\sqrt{\sum_{i=1}^{k}n_i(n_i-1)} \geqs \sqrt{\sum_{i=2}^{k}n_i}.
\end{equation}

\vs

Since $G = \prod_{i=1}^{k}{\rm SO}_{n_i}$ and $l(G) \leqs \b(\sqrt{\dim G}+c)$, the claim implies that $\sum_{i=2}^{k}n_i\leqs (\b c)^2$ and thus the normal subgroup ${\rm SO}_{n_1}$ has $c$-bounded codimension. Therefore, it suffices to prove the number-theoretic claim.

First assume the $n_i$ are all equal. We need to show that
\[
\frac{5}{4}kn_1 - \frac{7}{4}k-\frac{5}{4}\sqrt{kn_1(n_1-1)} - \sqrt{(k-1)n_1} \geqs 0
\]
for $(n_1,k) \ne (7,2)$. If $k$ is fixed, the expression on the left hand side is increasing in $n_1$ and it is routine to verify the desired bound.

Now assume that at least one $n_i$ is less than $n_1$, say $n_2<n_1$, and set
\[
f(n_2, \ldots, n_k) = \frac{5}{4}\sum_{i=1}^{k}n_i -\frac{7}{4}k - \frac{5}{4}\sqrt{\sum_{i=1}^{k}n_i(n_i-1)} - \sqrt{\sum_{i=2}^{k}n_i}
\]
so it suffices to show that $f(n_2, \ldots, n_k) \geqs 0$. Note that if $x,y,z \geqs 0$ are real numbers and $y \leqs z^2+2z\sqrt{x}$, then $\sqrt{x+y} \leqs z+\sqrt{x}$. Therefore, since we have
\[
\frac{16}{25}+\frac{8}{5}\sqrt{\sum_{i=1}^{k}n_i(n_i-1)} \geqs \frac{16}{25}+\frac{8}{5}\sqrt{2n_2(n_2-1)} \geqs 2n_2
\]
and
\[
\frac{1}{16}+\frac{1}{2}\sqrt{\sum_{i=2}^kn_i} \geqs \frac{1}{16}+\frac{1}{2}\sqrt{7} \geqs 1,
\]
it follows that
\[
\sqrt{2n_2+\sum_{i=1}^{k}n_i(n_i-1)} \leqs \frac{4}{5}+ \sqrt{\sum_{i=1}^{k}n_i(n_i-1)}
\]
and
\[
\sqrt{1+\sum_{i=2}^{k}n_i} \leqs \frac{1}{4} + \sqrt{\sum_{i=2}^{k}n_i}.
\]
These bounds imply that $f(n_2+1, n_3, \ldots, n_k) \geqs f(n_2, \ldots, n_k)$, so $f$ is minimal when $n_i=7$ for all $2 \leqs i \leqs k$. Finally, we note that
\[
f(7, \ldots, 7) = \frac{5}{4}n_1 + 7k - \frac{35}{4} - \frac{5}{4}\sqrt{n_1(n_1-1)+42(k-1)} - \sqrt{7(k-1)}
\]
is an increasing function in both $n_1$ and $k$, and by setting $(n_1,k)=(8,2)$ we see that $f(7, \ldots, 7) > 0$. This justifies the bound in \eqref{e:sum} and the proof of Theorem \ref{smalll} is complete.

\vs

In the other direction, if $G$ has dimension $d$, or has a normal subgroup isomorphic to ${\rm SO}_n$ of codimension $d$,
then $l(G) \leqs \b(\sqrt{\dim G} + c)$, where $c$ is $d$-bounded; indeed, this follows easily from Lemma \ref{lendim}.

\subsection{Proof of Theorem \ref{liedep}}

This is very similar to the proof of the analogous result for complex simple Lie groups (see \cite[Theorem 4]{bls-alg}). Let $G$ be a compact simple Lie group of rank $r$. It will be convenient to adopt the Lie notation for classical groups, so that $A_r = {\rm SU}_{r+1}$ and so on. By Lemma \ref{l:depth12}, we have $\l(G) \geqs 2$, with equality if and only if $G = {\rm SU}_{2}$, so we may assume $r \geqs 2$. Note that $\l(G) \geqs 3$, with equality if and only if $G$ has a maximal $A_1$ subgroup.

If $G = C_r$, then by applying \cite{Dynkin} and Corollary \ref{maxs}, we see that $G$ has a maximal $A_1$ subgroup and thus $\l(G)=3$. Next assume $G=B_r$, with $r \geqs 3$. If $r \geqs 4$ then $\l(G)=3$. However, if $r=3$ then $G$ does not have a maximal $A_1$ subgroup, so $\l(G) \geqs 4$. In this case, equality holds since
\[
B_3>G_2>A_1>T_1>1
\]
is an unrefinable chain. Similarly, by arguing as in the proof of \cite[Theorem 4]{bls-alg}, we see that $\l(G)=4$ if $G = D_{2r}$, or $A_r$ with $r \geqs 3$ and $r \ne 6$. Since $A_2$ has a maximal $A_1$ subgroup, we have $\l(A_2)=3$, so to complete the proof of Theorem \ref{liedep} for classical groups, we may assume that $G=A_6$. Here $\l(G) \leqs 5$ since $B_3$ is a maximal subgroup and $\l(B_3)=4$ as above. In addition, $\l(G) \geqs 4$ since $G$ does not have a maximal $A_1$ subgroup. Let $M$ be a maximal connected subgroup of $G$. By inspecting \cite{Dynkin}, we deduce that either $M = B_3$, or $M = A_5T_1$, $A_4A_1T_1$ or $A_3A_2T_1$ is the Levi factor of a maximal parabolic subgroup of $G$. If $M$ is a Levi factor, then $\l(M) \geqs \l(A_{k})$ for some $k \in \{3,4,5\}$ and we conclude that $\l(M) \geqs 4$ for each connected maximal subgroup $M$ of $G$. Therefore $\l(G)=5$.

Finally, if $G$ is an exceptional group, then we can repeat the argument in the proof of
\cite[Theorem 4]{bls-alg}. We omit the details.

\subsection{Proof of Theorem \ref{depbds}}

We first prove part (i). Let $G = S^k$, where $S$ is a compact simple Lie group. We proceed by induction on $k$, noting that the case $k=1$ is obvious. Assume $k\geqs 2$. Then $G$ has a maximal connected subgroup $M = D(S^2) \times S^{k-2} \cong S^{k-1}$, where $D(S^2)$ is a diagonal subgroup of $S^2$. By induction, $\l(M) = \l(S)+k-2$ and thus $\l(G) \leqs \l(M)+1 = \l(S)+k-1$, proving the required upper bound for $\l(G)$.

To establish the lower bound, let $\pi_i : G \to S$ be the $i$-th projection map and let $M$ be a maximal connected subgroup of $G$ such that $\l(M) = \l(G)-1$. If $\pi_i(M) = M_i < S$ for some $i$, then $M = M_i \times S^{k-1}$, so by induction $\l(M) \geqs \l(S^{k-1}) = \l(S)+k-2$ and hence $\l(G) \geqs \l(S)+k-1$. Otherwise, $\pi_i(M) = S$ for all $i$, so $M$ is a product of diagonal subgroups of various subsets of the simple factors of $S^k$, and maximality forces $M = D(S^2)\times S^{k-2} \cong S^{k-1}$. Hence again by induction we have
\[
\l(G) = \l(M)+1 = \l(S)+k-1,
\]
proving the lower bound. This establishes part (i).

Now consider part (ii), where $z = \dim Z(G)^0$ and $G' = \prod_{i=1}^m S_i^{k_i}$. The upper bound for $\l(G)$ follows from part (i) and Lemma \ref{l:easy}(ii). We now prove the lower bound by induction on $\sum_{i} k_i$. We have $\l(G) = \l(G')+z$ by Lemma \ref{triv}, so we may assume that $G = G'$. The case $\sum_{i} k_i=1$ is trivial, so assume $\sum_{i} k_i\geqs 2$. Let $M$ be a maximal connected subgroup of $G$ such that $\l(M) = \l(G)-1$. As above, without loss of generality, one of the following holds:
\begin{itemize}\addtolength{\itemsep}{0.2\baselineskip}
\item[(a)] $k_1\geqs 2$ and $M = D(S_1^2) \times  S_1^{k_1-2} \times \prod_{i=2}^m S_i^{k_i}$.
\item[(b)] $M = M_1 \times S_1^{k_1-1} \times \prod_{i=2}^m S_i^{k_i}$, where $M_1$ is maximal connected in $S_1$.
\end{itemize}
In case (a), induction gives $\l(M) \geqs (\sum_{i=1}^m(k_i+1))-1$ and the result follows. The same applies in case (b), unless $k_1=1$. In the latter case, let
$N = \prod_{i=2}^m S_i^{k_i}$, so $M/N \cong M_1$. Since $M_1 \ne 1$, an elementary argument (see \cite[Lemma 2.5]{bls-alg}) shows that $\l(M) \geqs \l(N)+1$. Induction
gives $\l(N) \geqs \sum_{i=2}^m (k_i+1)$ and hence $\l(G) = \l(M)+1 \geqs \sum_{i=1}^m(k_i+1)$, as required.

\subsection{Proof of Theorem \ref{ld}}

Let $G$ be a compact connected Lie group and write $G = G'Z(G)^0$ and $z = \dim Z(G)^0$. By Lemmas \ref{l:easy}(i) and \ref{triv}, we have $l(G) = l(G') + z$ and $\l(G) = \l(G')+z$. In particular, the result is trivial if $G$ is a torus, so assume that $G'\ne 1$.

By Theorem \ref{complen}, the only compact simple group satisfying $l(G) = \l(G)$ is $G = {\rm SU}_2$. Hence if $G'={\rm SU}_2$, then $l(G) = z+2 = \l(G)$.

Conversely, suppose that $l(G) = \l(G)$. Write $G' = \prod_{i=1}^{t}S_i$, a commuting product of simple groups $S_i$. Then $l(G) = z+\sum_i l(S_i)$ and $\l(G) \leqs z+\sum_i\l(S_i)$.
Hence $l(S_i) = \l(S_i)$ for all $i$, so $G' = ({\rm SU}_2)^t$ and $l(G)=z+2t$. It remains to show that $t=1$. To see this, suppose $t \geqs 2$ and note that there is an unrefinable chain
\[
({\rm SU}_2)^2 > D(({\rm SU}_{2})^2) > T_1>1,
\]
so $\l(({\rm SU}_2)^2) = 3$. Therefore, $\l(G) \leqs z+2t-1 < l(G)$, a contradiction. This completes the proof.

\subsection{Proof of Theorem \ref{cd}}

By Theorems \ref{complen} and \ref{liedep}, we see that ${\rm SU}_{3}$ is the only compact simple Lie group with chain difference one. It follows easily that the compact semisimple  Lie groups with chain difference one are ${\rm SU}_3$, $({\rm SU}_2)^2$ and ${\rm SU}_3 {\rm SU}_2$. The rest of the argument is very similar to the proof of Theorem \ref{ld} above.

\subsection{Proof of Theorem \ref{lcd}}\label{ss:cd}

We start with some preparations.

\begin{lem}\label{twice}
Let $S$ be a compact simple Lie group. Then
\[
l(S) \leqs 2 \cd(S) +a,
\]
where $a=2$ if $S = {\rm SU}_2, {\rm SU}_3, {\rm SU}_4$; $a=1$ if $S = {\rm Sp}_4, {\rm SO}_7$; and $a=0$ in all other cases.
\end{lem}

\begin{proof}
It suffices to show that $l(S) \geqs 2\l(S)-a$, which is easily deduced from
Theorems \ref{complen} and \ref{liedep}.
\end{proof}

Next, we deal with homogeneous semisimple groups.

\begin{lem}\label{homog}
Let $S$ be a compact simple Lie group and let $k \geqs 2$. Then
\[
l(S^k) \leqs 2 \cd(S^k),
\]
unless $S \cong {\rm SU}_2$, in which case $l(S^k) = 2 \cd(S^k) + 2$.
\end{lem}

\begin{proof}
We have $l(S^k) = k l(S)$ and $\l(S^k) = k + \l(S)-1$ (see Theorems \ref{general} and \ref{depbds}).
Combining these equalities with the values of $l(S)$ and $\l(S)$ (see Theorems \ref{complen} and
\ref{liedep}), we easily obtain the required conclusion.
\end{proof}

\begin{lem}\label{reduce2}
Let $G$ be a compact connected Lie group with $G' = \prod_{i=1}^m S_i^{k_i}$,
where the $S_i$ are pairwise non-isomorphic simple groups and $k_i \geqs 0$. Then
\[
\cd(G) = \cd(G') \geqs \sum_{i=1}^m \cd(S_i^{k_i}).
\]
\end{lem}

\begin{proof} This follows from Lemmas \ref{l:easy} and \ref{triv}.
\end{proof}

We can now prove Theorem \ref{lcd}.

Lemma \ref{reduce2} enables us to reduce to the case where $G$ is semisimple.
Write $G = \prod_{i=1}^m S_i^{k_i}$ where $m \geqs 5$, $k_i \geqs 0$ and the $S_i$ are pairwise non-isomorphic simple groups as in \eqref{e:simples}, labelled so that $S_1, \ldots , S_5$ are ${\rm SU}_2, {\rm SU_3}, {\rm SU}_4, {\rm Sp}_4, {\rm SO}_7$, respectively.
Set
\[
G_1 = S_1^{k_1} \times \prod_{2 \leqs i \leqs 5, \, k_i = 1} S_i,
\]
and
\[
G_2 = \prod_{2 \leqs i \leqs 5, \, k_i \geqs 2} S_i^{k_i} \times \prod_{i=6}^m S_i^{k_i}.
\]
Note that $G = G_1 \times G_2$, $l(G) = l(G_1) + l(G_2)$ and $\cd(G) \geqs \cd(G_1)+\cd(G_2)$, so it suffices to show that
\[
l(G) \leqs 2(\cd(G_1)+\cd(G_2))+2.
\]

Now
\[
l(G_2) = \sum_{2 \leqs i \leqs 5, \, k_i \geqs 2} l(S_i^{k_i}) + \sum_{i=6}^m l(S_i^{k_i}).
\]
By Lemma \ref{homog} we have $l(S_i^{k_i}) \leqs 2 \cd(S_i^{k_i})$ if $2 \leqs i \leqs 5$ and $k_i \geqs 2$.
Similarly, Lemma \ref{twice} gives $l(S_i) \leqs 2 \cd(S_i)$ for $6 \leqs i \leqs m$, which yields
$l(S_i^{k_i}) = k_i l(S_i) \leqs 2k_i \cd(S_i) \leqs 2 \cd(S_i^{k_i})$ for these values of $i$.
Applying Lemma \ref{reduce2}, we conclude that
\[
l(G_2) \leqs 2 \cd(G_2).
\]
Therefore, to complete the proof of the main statement of Theorem \ref{lcd}, it remains to show that
\begin{equation}\label{e:g1}
l(G_1) \leqs 2\cd(G_1)+2.
\end{equation}

First observe that
\[
l(G_1) = 2k_1+4k_2+6k_3+5k_4+7k_5
\]
and let $\gamma \geqs 0$ be the number of non-zero $k_i$ with $2 \leqs i \leqs 5$. It will be useful to highlight the following unrefinable chains (the existence of these chains follows by combining \cite{Dynkin} and Lemma \ref{corr}):
\[
{\rm SU}_{3} > {\rm SU}_{2}, \; {\rm SU}_{4} > {\rm Sp}_{4} > {\rm SU}_{2}, \;
{\rm SO}_{7} > {\rm SU}_{4}, \; {\rm SO}_{7} > G_2 > {\rm SU}_{2}.
\]

If $\gamma=0$ then the bound in \eqref{e:g1} follows from Lemma \ref{homog} (it is trivial if $k_1=0$ or $1$). For $\gamma>0$, we make use of the above unrefinable chains to bound $\l(G_1)$. Suppose $\gamma=1$, say
$G_1 = ({\rm SU}_{2})^{k_1} \times {\rm SU}_{3}$. Then $l(G_1) = 2k_1+4$ and $({\rm SU}_{2})^{k_1+1}<G_1$ is a maximal connected subgroup, so Theorem \ref{depbds}(i) yields $\l(G_1) \leqs k_1+3$ and thus $l(G_1) \leqs 2\cd(G_1)+2$. The other cases with $\gamma=1$ are just as easy and we omit the details. Similar reasoning applies when $\gamma>1$. For example, suppose $G_1 = ({\rm SU}_{2})^{k_1} \times {\rm SU}_{3} \times {\rm SO}_{7}$. Here $l(G_1) = 2k_1+11$ and there is an urefinable chain
\[
G_1 > ({\rm SU}_{2})^{k_1+1} \times {\rm SO}_{7} > ({\rm SU}_{2})^{k_1+1} \times G_2 > ({\rm SU}_{2})^{k_1+2},
\]
so $\l(G_1) \leqs 3+2+k_1+1 = k_1+6$, $\cd(G_1) \geqs k_1+5$ and thus $l(G_1) \leqs 2\cd(G_1)+1$. Similarly, if $\gamma=4$ then $l(G_1) = 2k_1+22$ and
\begin{align*}
G_1 > ({\rm SU}_{2})^{k_1} \times {\rm SU}_{3} \times ({\rm SU}_{4})^2 \times {\rm Sp}_{4} & > ({\rm SU}_{2})^{k_1} \times {\rm SU}_{3} \times {\rm SU}_{4} \times {\rm Sp}_{4} \\
& > ({\rm SU}_{2})^{k_1} \times {\rm SU}_{3} \times ({\rm Sp}_{4})^2 \\
& > ({\rm SU}_{2})^{k_1} \times {\rm SU}_{3} \times {\rm Sp}_{4} \\
& > ({\rm SU}_{2})^{k_1+1} \times {\rm Sp}_{4} \\
& > ({\rm SU}_{2})^{k_1+2}
\end{align*}
is unrefinable (here we are using the fact that a diagonal subgroup $D(S^2)<S^2$ is maximal for a simple group $S$), so $\l(G_1) \leqs 6+2+k_1+1 = k_1+9$ and $\cd(G_1) \geqs k_1+13$. In this way, one checks that the bound in \eqref{e:g1} holds and the proof of the first (and main) statement of Theorem \ref{lcd} is complete.

To prove the second statement, recall that, by Theorem \ref{dimlength} we have
\[
\b(\sqrt{\dim G'} - \a) \leqs l(G'),
\]
where the constants $\a, \b$ are defined as in Section \ref{ss:dl} (see \eqref{e:al} and \eqref{e:be}). Combining this with the first assertion of Theorem \ref{lcd}
we obtain
\[
\dim G/Z(G) = \dim G' \leqs (\b^{-1}l(G') + \a)^2  \leqs (\b^{-1}(2 \cd(G) + 2) + \a)^2.
\]
This completes the proof.

\end{document}